\documentclass[11pt]{article}
\usepackage{amsthm}
\usepackage{amsfonts}
\usepackage{amsmath}
\usepackage{amssymb}
\usepackage{graphicx}
\usepackage{epstopdf}
\usepackage{filecontents}
\usepackage[left=2.5cm,right=2.5cm,top=2.5cm,bottom=2.5cm]{geometry}

\newtheorem{proposition}{Proposition}
\newtheorem{corollary}{Corollary}
\newtheorem{theorem}{Theorem}[section]
\newtheorem{lemma}{Lemma}
\theoremstyle{definition}
\newtheorem{definition}{Definition}
\theoremstyle{remark}
\newtheorem*{remark}{Remark}
    \makeatletter
    \let\@fnsymbol\@arabic
    \makeatother

\title{Best $L_1$ approximation of Heaviside-type functions in Chebyshev and weak-Chebyshev spaces}%
\author{\setcounter{footnote}{0} Laurent Gajny\footnote{Arts et M\'etiers ParisTech, LSIS - UMR CNRS 7296, 8 Boulevard Louis XIV, 59046 Lille Cedex, France}, Olivier Gibaru$^{1,}$\footnote{INRIA Lille-Nord-Europe, NON-A research team, 40, avenue Halley 59650 Villeneuve d'Ascq}, Eric Nyiri$^1$, Shu-Cherng Fang\footnote{North Carolina State University, Edward P. Fitts Department of Industrial and Systems Engineering, NC 27695-7906, United States of America.}}%
\date{}%

\begin{document}
  \maketitle
\begin{abstract}
In this article, we study the problem of best $L_1$ approximation of Heaviside-type functions in Chebyshev and weak-Chebyshev spaces. We extend the Hobby-Rice theorem \cite{HobbyRice1965} into an appropriate framework and prove the unicity of best $L_1$ approximation of Heaviside-type functions in an even-dimensional Chebyshev space under the condition that the dimension of the subspace composed of the even functions is half the dimension of the whole space. We also apply the results to compute best $L_1$ approximations of Heaviside-type functions by polynomials and Hermite polynomial splines with fixed knots.\\

\end{abstract}

\section*{Introduction}

Let $[a,b]$ be a real interval with $a<b$ and $\nu$ be a positive measure defined on a $\sigma$-field of subsets of $[a,b]$. The space $L_1([a,b],\nu)$ of $\nu$-integrable functions with the so-called $L_1$ norm :
 \[\Vert \cdot \Vert_1 : f\in L_1([a,b],\nu)\longmapsto \Vert f \Vert_1 = \int_a^b \vert f(x)\vert \ \mathrm{d}\nu(x)\]
forms a Banach space. Let $f\in L_1([a,b],\nu)$ and $Y$ be a subset of $L_1([a,b],\nu)$ such that $f\notin \overline{Y}$ where $\overline{Y}$ is the closure of $Y$.  The problem of the best $L_1$ approximation of the function $f$ in $Y$ consists in finding a function $g^*\in Y$ such that :
\begin{equation}
  \Vert f-g^*\Vert_1 \le \Vert f-g\Vert_1, \quad \forall g\in Y.
\end{equation}
If $Y$ is a compact set or a finite dimensional subspace of $L_1([a,b],\nu)$, then such a function exists for any $f\in L_1([a,b],\nu)$. In this article, we focus on the finite-dimensional case. An important characterization theorem when $\nu$ is a non-atomic measure (See for example \cite{James1947,KripkeRivlin1965} and a development in \cite{Pinkus1989}) claims that finding a best $L_1$ approximation of a function $f$ remains to define a sign-changes function $s$ (\textit{i.e.} a function for which values alternate between -1 and 1 at certain abscissae) such that :
\[\int_a^b s(x)g(x)\ \mathrm{d}\nu(x)=0, \quad \forall g \in Y.\]
The Hobby-Rice theorem \cite{HobbyRice1965} enables to show the existence of such functions. This major result helps characterize and compute solutions of best $L_1$ approximation of a continuous function in Chebyshev and weak-Chebyshev spaces. It has been widely studied in the litterature \cite{Kammler1979,Micchelli1977,Peherstorfer1979,Rice1964,Strauss1984,Usow1967}. We recall that an $n$-dimensional Chebyshev space (resp. weak-Chebyshev space) on $[a,b]$ is a subspace of $C^0[a,b]$ composed of functions which have at most $n-1$ distinct zeros on $[a,b]$ (resp. strong changes of sign) \cite{Karlin1966,Schumaker1981}. Classical examples of such spaces are respectively polynomial functions and polynomial spline functions with fixed knots.\\
To the best of our knowledge, there are few references in the litterature that deal with the challenging problem of best $L_1$ approximation of discontinuous functions, especially Heaviside-type functions.

 \begin{definition}
We say that $f$ is a Heaviside-type function on $[a,b]$ with a jump at $\delta\in (a,b)$ if $f\in C^0([a,b]\backslash\{\delta\})$ and if the limits of the function $|f|$ on both sides of $\delta$ exist and are finite.\\
We denote by $\mathcal{J}_\delta[a,b]$ the vector space of such functions.
 \end{definition}

 Existing papers focusing on one or two-sided best $L_1$ approximation deal with specific spaces. Bustamante \emph{et al.} have shown that best polynomial one-sided $L_1$ approximation of the Heaviside function can be obtained by Hermite interpolation at zeros of some Jacobi polynomials \cite{Bustamante2012}. Moskona \emph{et al.} have studied the problem of best two-sided $L_p$ ($p\ge1$) approximation of the Heaviside function using trigonometrical polynomials \cite{Moskona1995}. Saff and Tashev have done a similar work using polygonal lines \cite{SaffTashev1999}. We propose in this article a more general framework for best (two-sided) $L_1$ approximation of Heaviside-type functions in Chebyshev and weak-Chebyshev spaces which includes the two previous cases. This problem is presented in Section 1. We evidence the encountered difficulties that lead us to extend the classical Hobby-Rice theorem. This extension is presented in details in Section 2. In the third section, we strongly use this result to characterize best $L_1$ approximations of Heaviside-type functions in a Chebyshev space. We give sufficient conditions on the Chebyshev space to obtain a unique $L_1$ best approximation for every Heaviside-type function. We apply these results to polynomial approximation. In Section 4, we study the problem of best $L_1$ approximation of Heaviside-type functions in a weak-Chebyshev space. In particular, this theory is applied to the Hermite polynomial spline case in Section 4.

\section{Best $L_1$ approximation}

In this section, we recall an important characterization theorem of best $L_1$ approximation (See for example \cite{James1947,KripkeRivlin1965}) and the Hobby-Rice theorem. We evidence the reason why the Hobby-Rice theorem fails to give a more precise characterization of best $L_1$ approximations of Heaviside-type functions.
For convenience, we define the zero set $Z(f)=\{x\in [a,b],\ f(x)=0\}$ and the sign function :
\[\text{sign} : x\in \mathbf{R} \longmapsto
\text{sign}(x)=\left\{ \begin{tabular}{rc}
$1$ & if $x>0$,\\
$0$ & if $x=0,$\\
$-1$& if $x<0.$\\
\end{tabular}\right.
\]
\begin{theorem}\label{th_caract}
  Let $Y$ be a subspace of $L_1([a,b],\nu)$ and $f\in L_1([a,b],\nu)\backslash \overline{Y}$. Then $g^*$ is a best $L_1$ approximation of $f$ in $Y$ if and only if :
  \begin{equation}\label{ineg_caract}
    \left|\int_a^b \mathrm{sign}(f(x)-g^*(x))g(x)\ \mathrm{d}\nu(x)\right| \le \int_{Z(f-g^*)} \vert g(x)\vert \ \mathrm{d}\nu(x), \quad \text{for all } g\in Y.
  \end{equation}
In particular, if $\nu(Z(f-g^*))=0$, then $g^*$ is a best $L_1$ approximation of $f$ in $Y$ if and only if :
  \begin{equation}\label{eg_caract}
    \int_a^b \mathrm{sign}(f(x)-g^*(x))g(x)\ \mathrm{d}\nu(x)=0,
  \end{equation}
  for all $g\in Y$.\\
\end{theorem}

The second part of the latter theorem claims that finding a best $L_1$ approximation of a function $f$ may remain to find a sign-changes function $s$, that alternates between $-1$ and $1$ at certain abscissae, such that:
\begin{equation}
\int_a^b s(x)g(x)\ \mathrm{d}\nu(x)=0, \quad \forall g \in Y. \label{scf}
\end{equation}
If there exists a function $g^*\in Y$ which interpolates $f$ at these abscissae and no others, then $g^*$ is a best $L_1$ approximation of $f$ in $Y$. Thus, the non-linear problem of best $L_1$ approximation remains to a simple Lagrange interpolation problem.\\
Actually, based on a result by Phelps (Lemma 2 in \cite{Phelps1966}), it can be shown that if $\nu$ is a non-atomic measure then $g^*$ is a best $L_1$ approximation of $f$ in $Y$ if and only if there exists a sign changes function satisfying \eqref{scf} and which coincides $\nu$-almost everywhere with $sgn(f-g^*)$ on $[a,b]\backslash Z(f-g^*)$ (see also Theorem 2.3 in \cite{Pinkus1989}). From now, we consider in this article that we are in this case. Thus, either if a best $L_1$ approximation coincides or not with its reference function on a set of non-zero measure, the determination of a sign-changes function verifying \eqref{scf} is fundamental. Indeed, it characterizes a best $L_1$ approximation in both cases. The Hobby-Rice theorem \cite{HobbyRice1965} enables to show the existence of such a sign-changes function with the number of sign-changes lower or equal to the dimension of $Y$.

\begin{theorem}[Hobby, Rice, 1965]
Let $Y$ be an $n$-dimensional subspace of $L_1([a,b],\nu)$ where $\nu$ is a finite, non-atomic measure. Then there exists a sign-changes function :
\begin{equation}\label{sign_avec_abs}
s(x)=\sum_{i=1}^{r+1}(-1)^i\mathbf{1}_{[\alpha_{i-1},\alpha_{i}[}(x),
\end{equation}
such that for all $g\in Y$ :
\begin{equation}
\label{eq_HR}
 \int_a^b s(x)g(x)\ \mathrm{d}\nu(x)=0,
\end{equation}
 where $r\le n$ and :
 \[ a=\alpha_0<\alpha_1<\alpha_2<\dots<\alpha_r<\alpha_{r+1}=b.\\\]
\label{th_hr}
\end{theorem}

\begin{remark}
A simple proof of this theorem can be found in \cite{Pinkus1976b} using the Borsuk antipodality theorem (See for example \cite{Nirenberg1974}, page 21).
\end{remark}

One cannot guarantee in general that a sign function with $r\le n$ changes of sign as stated in the Hobby-Rice theorem enables us to define a function $g^*$ that coincides with the function to approximate $f$ at these abscissae and no others. In Figure \ref{trop_dintersection}, we can observe a best $L_1$ approximation of a continuous function, $x\mapsto \sin(3x)$, in a 4-dimensional space - the space of cubic functions on the interval $[-1,1]$ - having five intersections with the above function. The following corollary partially explains this observation.

 \begin{figure}[!h]
   \centering
   \includegraphics[width=\linewidth]{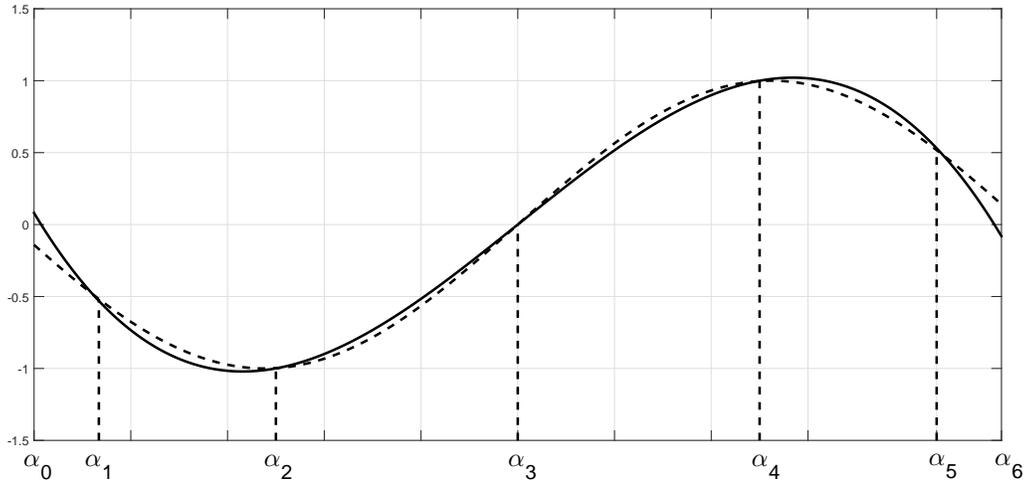}
   \caption{Best $L_1$ cubic approximation (solid line) of $x \mapsto \sin(3x)$ (dashed line) on $[-1,1]$.}
   \label{trop_dintersection}
 \end{figure}

\begin{corollary}
  Let $Y$ be an $n$-dimensional subspace of $L_1([a,b],\nu)$ where $\nu$ is a finite, non-atomic measure and $m$ be an integer such that $m\ge n$. Then there exist a sign-changes function :
\begin{equation}s(x)=\sum_{i=1}^{r_m+1}(-1)^i\mathbf{1}_{[\alpha_{i-1},\alpha_{i}[}(x),
\end{equation}
such that for all $g\in Y$ :
\begin{equation}
  \int_{a}^{b} s(x)g(x)\ \mathrm{d}\nu(x) =0,
\end{equation}
 where $r_m\le m$ :
 \[ a=\alpha_0<\alpha_1<\alpha_2<\dots<\alpha_{r_m}<\alpha_{r_m+1}=b.\]
\label{Cor_Hr}
\end{corollary}

\begin{proof}
  Complete $Y$ in a $m$-dimensional subspace of $L_1([a,b],\nu)$. Then apply the Hobby-Rice theorem.
\end{proof}

The Hobby-Rice theorem is very useful to characterize best $L_1$ approximations of continuous functions in Chebyshev and weak-Chebyshev spaces. In particular in an $n$-dimensional Chebyschev space, it can be shown that there exists a unique set of $n$ real values, sometimes called canonical points, which satisfies \eqref{eq_HR}. This leads to the following classical result (see for example \cite{KreinNudelman1977}, page 339-340 or \cite{Nurnberger2013}, page 73) that proves that for a class of continuous functions, best $L_1$ approximation can be obtained by Lagrange interpolation.

\begin{theorem}
 Let $Y$ be an $n$-dimensional Chebyshev space on $[a,b]$. Then any continuous function on $[a,b]$ has a unique best $L_1$ approximation in $Y$.\\
 Moreover, if $span(Y,f)$ is a $(n+1)$-dimensional Chebyshev space on $[a,b]$ then $g^*$, the best $L_1$ approximation of $f$ in $Y$, is fully determined by the Lagrange interpolation conditions :
  \[g^*(\alpha_i)=f(\alpha_i),\quad i=1,\dots,n,\]
  where the $\alpha_i$ values are those given in the Hobby-Rice theorem.
  \label{th_unicite_cheby}
\end{theorem}

There is an analogous result for weak-Chebyshev spaces proved by Michelli \cite{Micchelli1977}. Our goal is to prove similar results for the problem of best $L_1$ approximation of Heaviside-type functions. This case is different because a change of sign in the difference function between the function to approximation $f$ and a best $L_1$ approximation $g^*$ must occur at the discontinuity $\xi_0$ (see Figure \ref{fig_P3_0}). However the Hobby-Rice theorem cannot guarantee that one of the $\alpha_i$ is equal to $\xi_0$. So we decide to look for a solution of \eqref{eq_HR} with one of the $\alpha_i$ values fixed at $\xi_0$. This will be the object of the next section.

\begin{figure}[!h]
\centering
\includegraphics[width=\linewidth]{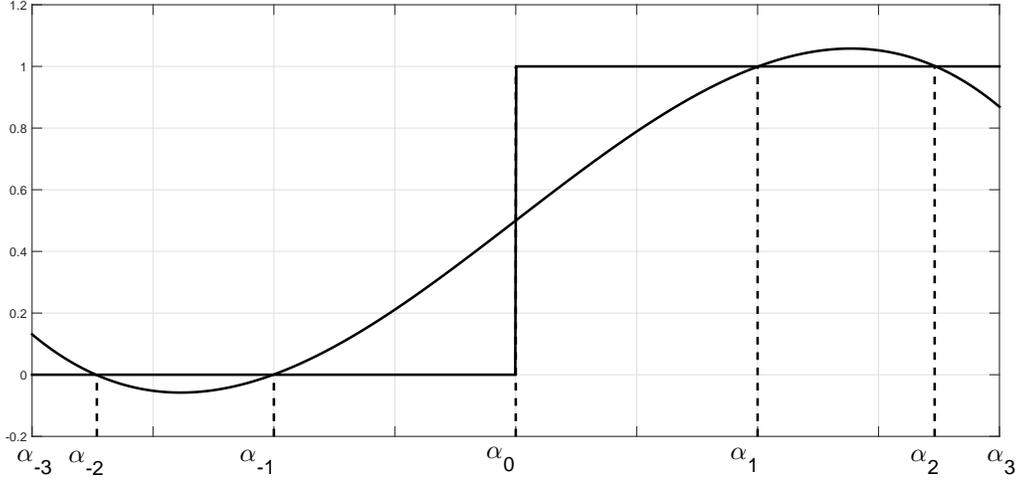}
\caption{Best $L_1$ cubic polynomial approximation of the Heaviside function.}
\label{fig_P3_0}
\end{figure}

\section{Extented Hobby-Rice theorem}

Without loss of generality, we give this result on $L_1([-1,1],\nu)$ and with $\xi_0=0$. Indeed, we can show that $\mathcal{J}_{\xi_0}[a,b]$ and $\mathcal{J}_0[-1,1]$ are homeomorphic by mean of an homography. This enables us to establish a symmetric version of the Hobby-Rice theorem by eliminating the even functions. This new version will be essential to characterize solutions of best $L_1$ approximation of Heaviside-type functions in Chebyshev and weak-Chebyshev spaces.\\

\begin{theorem}
  \label{th_hr_extended}
 Let $Y$ be an $n$-dimensional subspace of $L_1([-1,1],\nu)$ where $\nu$ is a finite, non-atomic measure. Let $Z$ be the subspace of $Y$ composed of the even functions and $q=n-\mathrm{dim }\ Z$. Then for any fixed $m\ge q$, there exists a sign-changes function :
 \begin{equation}
s(x)=\sum_{i=-r_{m}}^{r_{m}+1}(-1)^i\mathbf{1}_{[\alpha_{i-1},\alpha_{i}[}(x).
 \end{equation}
 such that for any $g\in Y$ :
\begin{equation}
\label{eq_HR_sym}
\int_{-1}^1 s(x)g(x)\ \mathrm{d}\nu(x)=0,
\end{equation}
 where :
  \[-1=\alpha_{-r_{m}-1}<\alpha_{-r_{m}}<\dots<\alpha_{-1}<\alpha_0=0<\alpha_1<\dots<\alpha_{r_{m}}<\alpha_{r_{m}+1}=1,\]
with $r_m\le m$, $\alpha_{-i}=-\alpha_i$, $i=1,\dots,r_{m}.$\\
\end{theorem}

\begin{proof}
 The idea of the proof is to find the positive part of the real sequence  $\{\alpha_i$, $i=1,\dots,r_m\}$ by applying the Corollary \ref{eq_HR_sym} to a well-chosen space $\widehat Y\subset L_1([0,1], \nu)$ of dimension  $\widehat q\leq q$. Then, we deduce \eqref{eq_HR_sym} by symmetry. \\
Firstly, let us not that for any odd sign-changes function $s$ and for any $g\in Y$, we have :
\[
\int_{-1}^1 s(x)g(x) d\nu(x)=\int_{0}^1s(x)(g(x)-g(-x)) d\nu(x).
\]
This proves that the equality \eqref{eq_HR_sym} is automatically satisfied for any $g\in Z$. Define the space $W$ of dimension $q$ such that $Z \oplus W= Y$. The problem is to find a sign-changes function $s$ such that :
\[
\int_{0}^1 s(x)(g(x)-g(-x)) d\nu(x)=0  \hbox{ for all }g\in W.
\]
This property is obtained by applying the Corollary \ref{Cor_Hr} to the space $\widehat Y$ (of dimension at most $q$) defined as follows :
\[
\widehat Y=\{\widehat g : x \in [0,1]\mapsto g(x)-g(-x),\ g\in W\}.\]
Finally, we obtain an odd sign-changes function satisfying \eqref{eq_HR_sym}.
\end{proof}

\section{Best $L_1$ approximation of Heaviside-type functions in a Chebyshev space}

In this section, we firstly give a lemma about the dimension of the subspace of an $n$-dimensional Chebyshev space on $[-1,1]$ composed of the even functions. This lemma will be very helpful in giving a version of the extended Hobby-Rice theorem in a Chebyshev space.

\begin{lemma}
  Let $Y$  be an $n$-dimensional Chebyshev space on $[-1,1]$. The dimension of the subspace of $Y$ composed of the even functions is at most $\lceil n/2 \rceil$ where $\lceil \cdot \rceil$ is the ceiling function.
\end{lemma}

\begin{proof}
Assume to the contrary that $Z$, the subspace of $Y$ composed of the even functions, has dimension $\lceil n/2 \rceil+1$. One can construct an even function with $\lceil n/2 \rceil$ zeros on the interval $(0,1)$. Since this function is even, it has  $2\lceil n/2 \rceil$ zeros on the interval $[-1,1]$. This is impossible because $Y$ is an $n$-dimensional Chebyshev space on $[-1,1]$.

\end{proof}

\noindent Now we can give a version of the extended Hobby-Rice theorem in a Chebyshev space.

\begin{proposition}
  Let $Y$ be an $n$-dimensional Chebyshev space on $[-1,1]$ and $Z$ be the subspace of $Y$ composed of the even functions. Assume $\mathrm{dim} \ Z = \lceil n/2 \rceil$. Then there exists a unique sign-changes function :
 \begin{equation}
s(x)=\sum_{i=-\lfloor n/2 \rfloor}^{\lfloor n/2 \rfloor+1}(-1)^i\mathbf{1}_{[\alpha_{i-1},\alpha_{i}[}(x).
 \end{equation}
   such that for any $g\in Y$ :
\begin{equation}
  \int_{-1}^1 s(x)g(x)\ \mathrm{d}\nu(x) =0, \quad j=1,2,\dots,n,
  \label{eq_HR_cheby2}
\end{equation}
where :
\begin{equation*}-1=\alpha_{-\lfloor n/2 \rfloor-1}<\alpha_{-\lfloor n/2 \rfloor}<\dots<\alpha_{-1}<\alpha_0=0<\alpha_1<\dots<\alpha_{\lfloor n/2 \rfloor+1}=1,
\end{equation*}
with $\alpha_{-i}=-\alpha_i$, $i=1, \dots,\lfloor n/2 \rfloor$ .\\
\label{prop_hr_cheby2}
\end{proposition}

\begin{proof}
We apply the Theorem \ref{th_hr_extended}, page \pageref{th_hr_extended} for $m=\lfloor n/2 \rfloor$. We obtain a real sequence $\{\alpha_i,\ i=-r_m,\dots,r_m\}$ with $\alpha_0=0$ and $\alpha_{-k}=-\alpha_k$, $k=1,\dots,r_m$ such that \eqref{eq_HR_sym} is satisfied. Assume $r_m<\lfloor n/2 \rfloor$.\\
Since $Y$ is a Chebyshev space, there exists $g\in Y$ that vanishes and changes of sign at the $\alpha_i$ on $[-1,1]$.\\
Then the function $x\mapsto s(x)g(x)$ has constant sign on $[-1,1]$ and :
\[\int_{-1}^1 s(x)g(x)d\nu(x)\neq 0.\]
This contradicts \eqref{eq_HR_cheby2}.\\
Assume now that there exists another sequence of real values not all equals to the $\alpha_i$ satisfying \eqref{eq_HR_cheby2} and denoted by $\{\beta_i,\ i=-\lfloor n/2 \rfloor,\dots,\lfloor n/2 \rfloor\}$. We have $\beta_0=\alpha_0=-1$ and $\beta_{\lfloor n/2 \rfloor+1}=\alpha_{\lfloor n/2 \rfloor+1}=1$. Define the two following sign-changes functions:
\[s_\alpha(x)=\sum_{i=-\lfloor n/2 \rfloor}^{\lfloor n/2 \rfloor+1}(-1)^i\mathbf{1}_{[\alpha_{i-1},\alpha_{i}[}(x),\]
\[s_\beta(x)=\sum_{i=-\lfloor n/2 \rfloor}^{\lfloor n/2 \rfloor+1}(-1)^i\mathbf{1}_{[\beta_{i-1},\beta_{i}[}(x).\]
This proof of unicity is derived from the one given in \cite{Nurnberger2013}, page 61, for the classical version of the Hobby-Rice theorem. Set :
\[\alpha_p=\min_{i} \{\alpha_i \ | \ \alpha_i\neq \beta_i\}.\]
We may assume that $\alpha_p<\beta_p$. Otherwise, we exchange the roles of $\alpha_p$ and $\beta_p$. It follows that :
\begin{align*}
  s_\alpha(x)-s_\beta(x) & = 0, \quad x\in [-1,\alpha_p),\\
  (-1)^p( s_\alpha(\alpha_p) -s_\beta(\alpha_p))&>0,\\
  (-1)^i( s_\alpha(x) -s_\beta(x))&\ge 0, \quad x\in [\alpha_i,\alpha_{i+1}],\ i=p,\dots,\lfloor n/2 \rfloor.
\end{align*}
Now, define the function :
\[g(x)=\mathrm{det}(\phi_k(\gamma_l))_{k=1\dots,n}^{l=1\dots,n},\quad x\in[-1,1],\]
where $(\gamma_1,\dots,\gamma_n)=(\alpha_1,\dots,\alpha_{j-1},\alpha_{j+1},\dots,x)$. Since $\{\phi_j\}_j$ is a Chebyshev system, by replacing $g$ by $-g$ if necessary, we have :
\[(-1)^ig(x)\ge 0, \quad x\in [\alpha_i,\alpha_{i+1}],\ i=p,\dots,\lfloor n/2 \rfloor,\]
and $(-1)^pg(\alpha_p)>0.$ We then have :
\[g(x)(s_\alpha(x)-s_\beta(x))\ge 0, \quad x\in [\alpha_i,\alpha_{i+1}],\ i=p,\dots,\lfloor n/2 \rfloor,\]
and $g(\alpha_p)(s_\alpha(\alpha_p)-s_\beta(\alpha_p))> 0.$ This contradicts the fact that both of $\{\alpha_i\}$ and $\{\beta_i\}$ are solutions of \eqref{eq_HR_cheby2}.
\end{proof}

\noindent We are now able to show the following unicity theorem.

\begin{theorem}
  Let $Y$ be an $n$-dimensional Chebyshev space on $[-1,1]$ where $n$ is even. Assume that the subspace of $Y$ composed of the even functions has dimension $n/2$. Then every function $f\in \mathcal{J}_0[-1,1]$ has a unique best $L_1$ approximation in $Y$.\\
  Moreover,
  \begin{equation}
  \text{if every element of } \mathrm{span}(Y,f) \text{ has a most } n \text{ simple zeros, }
  \label{prop_etoile}
  \end{equation}
  then the best $L_1$ approximation $L_1$ $g^*$ of $f$ in $Y$ is fully determined by :
  \[g^*(\alpha_i)=f(\alpha_i),\quad  i=-n/2,-n/2+1,\dots,-1,1,\dots,n/2,\]
  where the $\alpha_i$ are those given in Proposition \ref{prop_hr_cheby2}.\\
  \label{th_cheby_saut}
\end{theorem}

\begin{proof}
The scheme of the unicity proof is derived from a proof of Rice in \cite{Rice1964Book} (Theorem 4.4, page 109). Assume that $g_1$, a best $L_1$ approximation of $f$ in $Y$, is such that $\nu(Z(f-g_1))=0$. Since $n$ is even, by the Proposition \ref{prop_hr_cheby2}, $Z(f-g_1)$ contains at least $n$ elements. Assume also that there exists another best $L_1$ approximation $g_2$. Then, $Z(f-g_2)$ contains at least $n$ elements. We must have (see \cite{Pinkus1989}, page 17) :
\[(f-g_1)(f-g_2)\ge 0, \quad \nu-\text{a.e. on } [-1,1].\]
If there exists $x\in Z(f-g_1)$ such that $g_1(x)\neq g_2(x)$ then $f(x)-g_2(x)\neq 0$. The last inequality cannot hold almost everywhere on $[-1,1]$. So $g_1$ and $g_2$ must be equal on $Z(f-g_1)$, \textit{i.e.}  at at least $n$ points. But $Y$ is $n$-dimensional Chebyshev space on $[-1,1]$ then $g_1-g_2=0$.\\
Assume now that $f$ has two best $L_1$ approximations, $g_1$ and $g_2$ such that $\nu(Z(f-g_1))\neq 0$ and $\nu(Z(f-g_2))\neq 0$. For any $\lambda \in [0,1]$, $\lambda g_1+(1-\lambda)g_2$ is also a best $L_1$ approximation of $f$. Let the function $\psi$ be defined on $[0,1]\times[-1,1]$ by
\[ \psi(\lambda,x)=f(x)-\lambda g_1(x)-(1-\lambda)g_2(x).\]
If there exists $\lambda\in [0,1]$ such that $\nu(Z(\psi(\lambda,\cdot))=0$ then we conclude by the first part of the proof.\\
We assume now that for all $\lambda\in [0,1]$, $\nu(Z(\psi(\lambda,\cdot))>0$. We use Lemma 4.7. from \cite{Rice1964Book} page 120. It claims that there exist $\lambda_1$ and $\lambda_2$ such that $\nu(Z(\psi(\lambda_1,\cdot))\cap Z(\psi(\lambda_2,\cdot))\neq 0$.\\
In particular, $Z(\psi(\lambda_1,\cdot))\cap Z(\psi(\lambda_2,\cdot))$ contains at least $n$ points and so :
\[ \lambda_1g_1+(1-\lambda_1)g_2=\lambda_2g_1+(1-\lambda_2)g_2.\]
It comes $g_1=g_2$.\\
We apply the Proposition \ref{prop_hr_cheby2}. There exists a unique sequence of $n$ real values $\alpha_i$ such \eqref{eq_HR_cheby2} holds. From \cite{Zielke1979}, Chapitre 1, page 1, we know that there exists a function $g^*\in Y$ which vanishes at the $\alpha_i$. Since every non trivial element of $span(Y,f)$  has at most $n$ simple zeros, $g^*-f$ vanishes only at the $\alpha_i$. The fonction $g^*$ is the best $L_1$ approximation $L_1$ of $f$ in $Y$.
\end{proof}

 We consider now best $L_1$ approximation by an algebraic polynomial function of degree $n-1$ in $[-1,1]$. The space of such functions is an $n$-dimensional Chebyshev space on $[-1,1]$. Moreover, its canonical base has $\lceil n/2 \rceil$ even functions. So we can apply the Proposition \ref{prop_hr_cheby2}. Moreover we can explicitely determine the $\alpha_i$ values. One can show that for any integer $m \ge \lfloor n/2 \rfloor$, we have :
\begin{equation}
  \forall P \in \mathbf{P}_{n-1}[-1,1], \ \sum_{i=-m}^{m+1}(-1)^i \int_{\alpha_{i-1}}^{\alpha_i} P(x)\ \mathrm{d}x =0,
  \label{eq_HR_cheby_pol}
\end{equation}
for :
\begin{equation}
  \alpha_i=\cos\left(\frac{(m+1-i)\pi}{2m+2}\right),\quad i=-m-1,\dots,m+1.
\end{equation}
\label{prop_HR_pol}
These $\alpha_i,\ i=-m,\dots,m$ are the zeros of the second order Chebyshev polynomial of degree $2m+1$ defined as follows :
  \begin{equation}
    T_{m+1}(\cos(\theta))=\frac{\sin((m+2)\theta)}{\sin(\theta)},\quad \theta \in [-\pi,\pi].
  \end{equation}

Under the condition \eqref{prop_etoile}, these values enable to determine best $L_1$ approximation of certain Heaviside-type functions by a polynomial function. We can for example determine best $L_1$ polynomial approximations of the Heaviside function on $[-1,1]$. In particular, we have illustrated in Figure \ref{fig_P3_0} the unique best cubic approximation of the Heaviside which interpolation these function at $\cos(i\pi/6)$, $i=1,\dots,5$. Indeed, one can show that $span(\mathbf{P}_3,H)$ satisfies the property \eqref{eq_HR_cheby_pol}. Assume to the contrary that there exists a cubic function $g\in\mathbf{P}_3$ such that $g-H$ has five simple zeros on $[-1,1]$. It is impossible that four or all of these zeros lie on $[-1,0]$ or $[0,1]$ because $g$ has at most three zeros. If $g-H$ vanishes three times on $[-1,0]$ (respectively [0,1]) and twice on $[0,1]$ (respectively [-1,0]), then $g$ has two inflexion points, which is impossible for a cubic function.

\section{Best $L_1$ approximation of Heaviside-type functions in a weak-Chebyshev space}

In this section, we first extend Proposition \ref{prop_hr_cheby2} to the case of weak-Chebyshev space.

\begin{proposition}
  Let $Y$ be an $n$-dimensional weak-Chebyshev space on $[-1,1]$ and $Z$ be the subspace of $Y$ composed of the even functions. Assume $\mathrm{dim} \ Z = \lceil n/2 \rceil$. Then there exists a sign-changes function :
 \begin{equation}
s(x)=\sum_{i=-\lfloor n/2 \rfloor}^{\lfloor n/2 \rfloor+1}(-1)^i\mathbf{1}_{[\alpha_{i-1},\alpha_{i}[}(x).
 \end{equation}
   such that for any $g\in Y$ :
\begin{equation}
  \int_{-1}^1 s(x)g(x)\ \mathrm{d}\nu(x) =0, \quad j=1,2,\dots,n,
  \label{eq_HR_wcheby2}
\end{equation}
where :
\[-1=\alpha_{-\lfloor n/2 \rfloor-1}<\alpha_{-\lfloor n/2 \rfloor}<\dots<\alpha_{-1}<\alpha_0=0<\alpha_1<\dots<\alpha_{\lfloor n/2 \rfloor+1}=1,\]
with $\alpha_{-i}=-\alpha_i$, $i=1, \dots,\lfloor n/2 \rfloor$.\\
\label{prop_hr_wcheby2}
\end{proposition}

\begin{proof}
We prove the result for $n$ being even. The odd case follows in a similar way. Since $\{\phi_j,\ j=1,2,\dots,n\}$ is a weak-Chebyshev system, there exists a sequence of Chebyshev systems $\{\phi_{j,p},\ j=1,2,\dots,n\}_{p\in\mathbf{N}}$ such that :
    \[\lim_{p\rightarrow +\infty} \Vert \phi_j-\phi_{j,p}\Vert_\infty= 0, \quad j=1,2,\dots,n.\]
    Moreover, we can define $\phi_{j,p}$, for $j=1,2,\dots,n$ and $p\in\mathbf{N}$, as follows (see for example \cite{Nurnberger2013}, page 84) :
    \[\phi_{j,p}(t)=\int_\mathbf{R} \phi_j(s)K_p(s,t) \ \mathrm{d}\nu(s),\]
    where $K_p$ is the Gauss kernel. Then if $\phi_j$ is even, so is $\phi_{j,p}$. Indeed, we have :
    \begin{align*}
      \phi_{j,p}(-t) & =\int_\mathbf{R} \phi_j(s)K_p(s,-t) \ \mathrm{d}\nu(s),\\
      & = \int_\mathbf{R} \phi_j(-s)K_p(-s,-t) \ \mathrm{d}\nu(s),\\
      & = \int_\mathbf{R} \phi_j(s)K_p(s,t) \ \mathrm{d}\nu(s),\\
      & = \phi_{j,p}(t).
    \end{align*}
    From Proposition \ref{prop_hr_cheby2}, we can find sequences $\{\alpha_{i,p}\}_{p\in\mathbf{N}}$ defined on $[-1,1]$ for $i=-m',\dots,m'$ with $\alpha_{-j,p}=\alpha_{j,p}$ for $j=1,\dots,m'$ and $p\in\mathbf{N}$  such that :
    \[
    \sum_{i=-m'}^{m'+1} (-1)^i\int_{\alpha_{i,p}}^{\alpha_{i+1,p}} \phi_{j,p}(x)\ \mathrm{d}\nu(x)=0, \quad \text{for } j=1,2,\dots,n.
    \]
    In particular, $\alpha_{0,p}=0$ and $\alpha_{m'+1,p}=1$ for all $p\in \mathbf{N}$.\\
    Since $[-1,1]$ is compact, for $i=1,2,\dots,n$, $\{\alpha_{i,p}\}_{m'\in\mathbf{N}}$ admits a convergent subsequence and we denote the limit by $\alpha_i$. It follows that :
        \begin{equation}
    \sum_{i=-m'}^{m'+1}(-1)^i \int_{\alpha_{i-1}}^{\alpha_i} \phi_j(x)\ \mathrm{d}\nu(x) =0, \quad j=1,2,\dots,n.
\end{equation}
The proposition follows.
\end{proof}

\noindent The next proposition is a direct consequence of Proposition \ref{prop_hr_wcheby2}.\\

\begin{proposition}
  Let $Y$ be $n$-dimensional weak-Chebyshev space on $[-1,1]$ such that $\mathrm{dim} Z = \lceil n/2 \rceil$ basis functions being even. A best $L_1$ approximation of a Heaviside-type function $f \in \mathcal{J}_0[-1,1]$ from $Y$ interpolates $f$ at at least $n$ points if $n$ is even and $n-1$ points if $n$ is odd.
\end{proposition}

\section{Application to Hermite polynomial splines with fixed knots}
Let us define firsty these polynomial splines.
\begin{definition}
A Hermite spline of order $k$ with nodes $\mathbf{x}=\{x_1<x_2<\dots<x_n\}$ is a $C^k$-continuous function which is a polynomial function of degree lower or equal to  $2k+1$ on each interval $[x_i,x_{i+1}]$. We denote $\tilde{\mathcal{S}}_{k,\mathbf{x}}$ (or simply $\tilde{\mathcal{S}}_k$) the vectorial space of such functions.
\end{definition}
\noindent A basis of $\tilde{\mathcal{S}}_k$ is given in the next proposition.
  \begin{proposition}
 Let $\mathbf{x}=\{x_1<x_2<\dots<x_n\}$. Consider the following $n(k+1)$ functions :
  \begin{equation}
    \begin{split}
      p_j(x) & =(x-x_1)^j, \quad j=0,1,\dots,2k+1,\\
      \varphi_{j,q}(x) & =(x-x_j)_+^q,\quad j=2,3,\dots,n-1;\ q=k+1,k+2,\dots,2k+1.
          \end{split}
          \label{basis}
  \end{equation}
   Then, the set $\{p_j\}_{j=0,\dots,2k+1}\cup \{\varphi_{j,q}\}_{j=1,\dots,n-1}^{q=k+1,\dots,2k+1}$  is a basis of $\tilde{\mathcal{S}}_{k,\mathbf{x}}$ and $\text{dim }\tilde{\mathcal{S}}_{k,\mathbf{x}}=n(k+1)$.\\
   \label{prop_basis}
\end{proposition}

\begin{proof}
  First, the functions $\{p_j\}_{j=0,\dots,2k+1}$ and  $\{\varphi_{j,q}\}_{j=1,\dots,n-1}^{q=k+1,\dots,2k+1}$ belong to $\tilde{\mathcal{S}}_k$. It is obvious for $\{p_j\}_{j=0,\dots,2k+1}$ since they are $C^\infty$ polynomial functions of degree lower or equal to $2k+1$. Moreover,
  \[\varphi_{j,q}^{(m)}(x_j)=0, \quad m=0,1,\dots,q-1,\quad q=k+1,k+2,\dots,2k+1,\]
  $\text{and } \varphi_{j,q}^{(q)}(x_j)=q!$. Then the functions $\{\varphi_{j,q}\}_{j=1:n-1}^{q=k+1:2k+1}$ belong to $\tilde{\mathcal{S}}_k$.\\
  We proceed now by induction on the number of knots. For convenience, we define $\mathbf{x}_k=\{x_1,\dots,x_k\}$, $k=1,\dots,n$.\\
   For $n=2$, $\{p_j\}_{j=0,\dots,2k+1}$ is the Taylor basis for $\tilde{\mathcal{S}}_{k,\mathbf{x}_2}$. We assume that the proposition is true for $n-1$ knots.\\
  Let $z\in\tilde{\mathcal{S}}_{k,\mathbf{x}_n}$, then : \[z_{|_{[x_1,x_{n-1}]}}\in\tilde{\mathcal{S}}_{k,\mathbf{x}_{n-1}}.\]
   By the induction hypothesis, there exist unique reals : \[\alpha_0,\dots,\alpha_{2k+1},\beta_{1,q},\dots,\beta_{n-2,q},\  q=k+1,\dots,2k+1,\]
   such that :
  \[
  \tilde{z}(x)=z_{|_{[x_1,x_{n-1}]}}(x)=\sum_{j=0}^{2k+1}\alpha_jp_j(x)+\sum_{j=1}^{n-2}\sum_{q=k+1}^{2k+1}\beta_{j,q}\varphi_{j,q}(x).
  \]
  Then we consider the function : \[\psi=z-\tilde{z}-\sum_{q=k+1}^{2k+1}\beta_{n-1,q}\varphi_{n-1,q},\] for some real values $\beta_{n-1,q}$. For $q=k+1,\dots,2k+1$, we have :
  \begin{align*}
    \psi^{(q)}(x_{n-1}^+) & =z^{(q)}(x_{n-1}^+)-\tilde{z}^{(q)}(x_{n-1}^+)-\beta_{n-1,q}q!\\
    & =z^{(q)}(x_{n-1}^+)-\tilde{z}^{(q)}(x_{n-1}^-)-\beta_{n-1,q}q!\\
    & =z^{(q)}(x_{n-1}^+)-z^{(q)}(x_{n-1}^-)-\beta_{n-1,q}q!.\\
  \end{align*}
  Then we set, for $q=k+1,\dots,2k+1$,
  \[\beta_{n-1,q}=\frac{z^{(q)}(x_{n-1}^+)-z^{(q)}(x_{n-1}^-)}{q!}.\]
  Hence, the property holds for $n$ and the proof is complete.
\end{proof}

\noindent We give an important property of the spline space $\tilde{\mathcal{S}}_k$.
\begin{proposition}
  For any $\mathbf{x}=\{x_1<\dots<x_n\}\in \mathbf{R}^n$ and $k \in \mathbf{N}$, the spline space $\tilde{\mathcal{S}}_{k,\mathbf{x}}$ is an $(n(k+1))$-dimensional weak-Chebyshev space on $[x_1,x_n]$.
\end{proposition}

\begin{proof}
  If $n=2$, $\tilde{\mathcal{S}}_{k,\mathbf{x}_2}=\mathbf{P}_{2k+1}$ which is a Chebyshev space and then is a weak-Chebyshev space.\\
  If $k=0$, we observe that $\tilde{\mathcal{S}}_{k,\mathbf{x}_n}=\mathcal{S}_{1,\mathbf{x}_n}$ which is a weak-Chebyshev space (see for example \cite{Nurnberger2013}, page 95).\\
  We assume now $k>1$, $n>2$ and by contradiction, that $s\in \tilde{\mathcal{S}}_{k,\mathbf{x}_n}$ has $n(k+1)$ changes of sign.
  By the Rolle's theorem, $s'$ has $n(k+1)-1$ changes of sign. Then inductively, $s^{(k)}$ has $n(k+1)-k=(n-1)k+n$ changes of sign.\\
  However, $s^{(k)}$ is a $C^0$ polynomial spline of degree $k+1$. Such a function has at most $(n-1)(k+1)=(n-1)k+n-1$ changes of sign. Thus we have a contradiction. 
\end{proof}

Unfortunately, the basis defined in \eqref{basis} does not satisfy the property of having at least half of its functions being even. However, we can show the following proposition.

\begin{proposition}
  Let $\mathbf{x}=\{-x_n<\dots<-x_1<0<x_1<\dots<x_n\}\in \mathbf{R}^{2n+1}$. Then there exists a basis of the vector space $\tilde{\mathcal{S}}_{k,\mathbf{x}}$ with at least half of its functions being even.
\end{proposition}

\begin{proof}
  Without loss of generality, we may consider $\mathbf{x}=\{-n,-n+1,\dots,n-1,n\}$. The dimension of $\tilde{\mathcal{S}}_{k,\mathbf{x}}$  is $(2n+1)(k+1)$. We would like to find $\lceil(2n+1)(k+1)/2\rceil$ functions being even and independent in $\tilde{\mathcal{S}}_{k,\mathbf{x}}$.
  Then we may apply the incomplete basis theorem.\\
  We prove the proposition when $k$ is odd. The odd case is similar. We consider firstly the $p_j$ functions. We can replace them by the canonical polynomial basis $1,x,\dots,x^{2k+1}$ and select $k+1$ even functions $1,x^2,\dots,x^{2k}$.\\
  Then, consider the functions $\varphi_{j,q}(x)$ with modifications, for $j=1,\dots,n-1$,
  \begin{equation*}
    \hat{\varphi}_{j,q}(x) = \left\{
    \begin{array}{ll}
      (x-j)^q  & \text{if } x\ge j,\\
      (-x-j)^q & \text{if } x\le -j.
    \end{array}
    \right. \quad q=k+1,k+2,\dots,2k+1.
  \end{equation*}
These $(k+1)(n-1)$ functions $\hat{\varphi}_{j,q}$ are even by construction and belong to $\tilde{\mathcal{S}}_{k,\mathbf{x}}$.\\
Finally, we consider the functions $\varphi_{0,q}(x)$ and we modify them as follows :
\begin{equation*}
    \hat{\varphi}_{0,q}(x) =\vert x \vert^q,\quad  \text{for } q \text{ being odd in } \{k+1,k+2,\dots,2k+1\} .
\end{equation*}
These $(k+1)/2$ functions $\hat{\varphi}_{0,q}$ are also even by construction and belong to $\tilde{\mathcal{S}}_{k,\mathbf{x}}$.\\
Therefore, we have defined \[(k+1)+(k+1)(n-1)+(k+1)/2=\frac{1}{2}(2n+1)(k+1)\]
 even functions of $\tilde{\mathcal{S}}_{k,\mathbf{x}}$. They are clearly linearly independent since they have different supports. We may conclude the proof by applying the incomplete basis theorem.
\end{proof}

By applying the results of the prior section, we evidence then best $L_1$ approximations of Heaviside-type functions using polynomial Hermite splines. A closed form of the real values of Proposition \ref{prop_hr_wcheby2} is not currently available. Further work will concentrate on the solvation of \eqref{eq_HR_wcheby2} in the polynomial Hermite spline case. However, we are able to compute them numerically and then define best $L_1$ approximation of the Heaviside function (see Figures \ref{fig_5nodes} and \ref{fig_7nodes}). As we said before, we may have more intersections with the Heaviside-type function than the dimension of the space. We evidence this case in Figure \ref{fig_oscillating_7nodes}. In this case, we have then fourteen intersections while the dimension of the cubic spline space is only ten. Moreover on some Heaviside-type functions, the repartition of the abscissae may not be symmetric. This will be the topic of future research.

\begin{figure}[!h]
  \centering
  \begin{minipage}{0.45\linewidth}
    \centering
    \includegraphics[width=\linewidth]{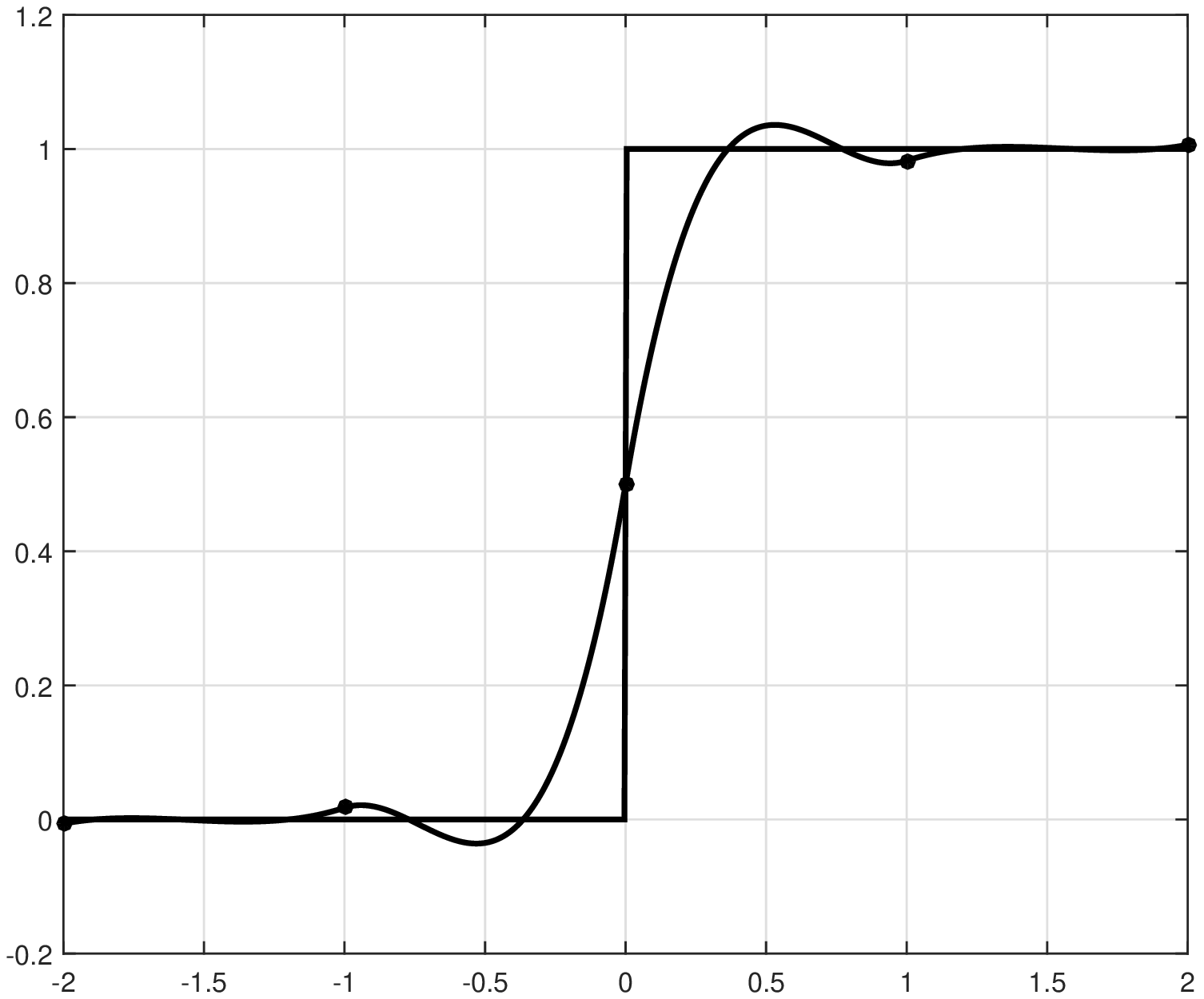}
  \end{minipage}
  \begin{minipage}{0.45\linewidth}
    \centering
    \includegraphics[width=\linewidth]{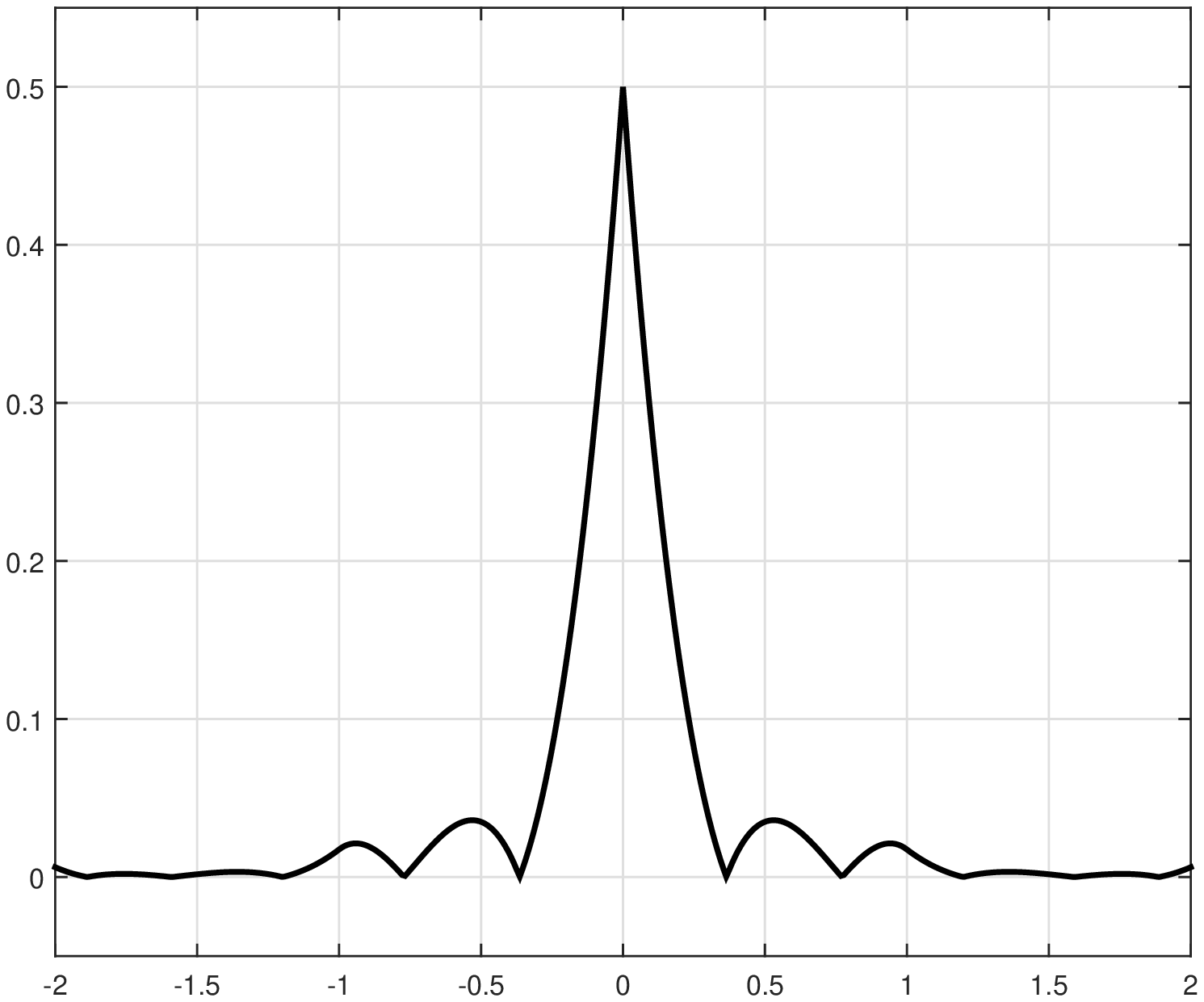}
  \end{minipage}
  \caption{Best $L_1$ approximation of the Heaviside function with five evenly spaced knots and pointwise error graph.}
  \label{fig_5nodes}
\end{figure}
\begin{figure}[!h]
  \centering
  \begin{minipage}{0.45\linewidth}
    \centering
    \includegraphics[width=\linewidth]{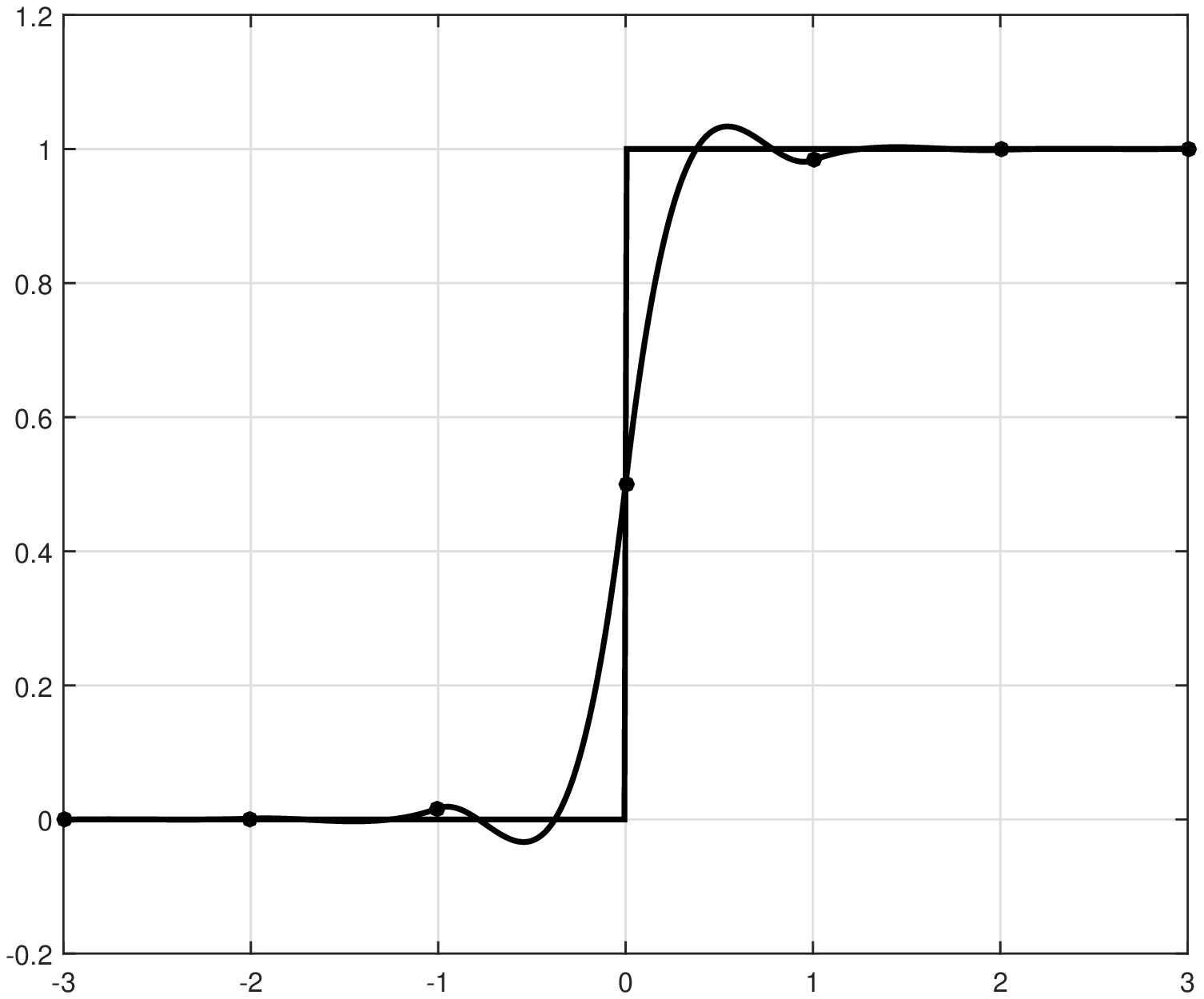}

  7-node solution
  \end{minipage}
   \begin{minipage}{0.45\linewidth}
    \centering
    \includegraphics[width=\linewidth]{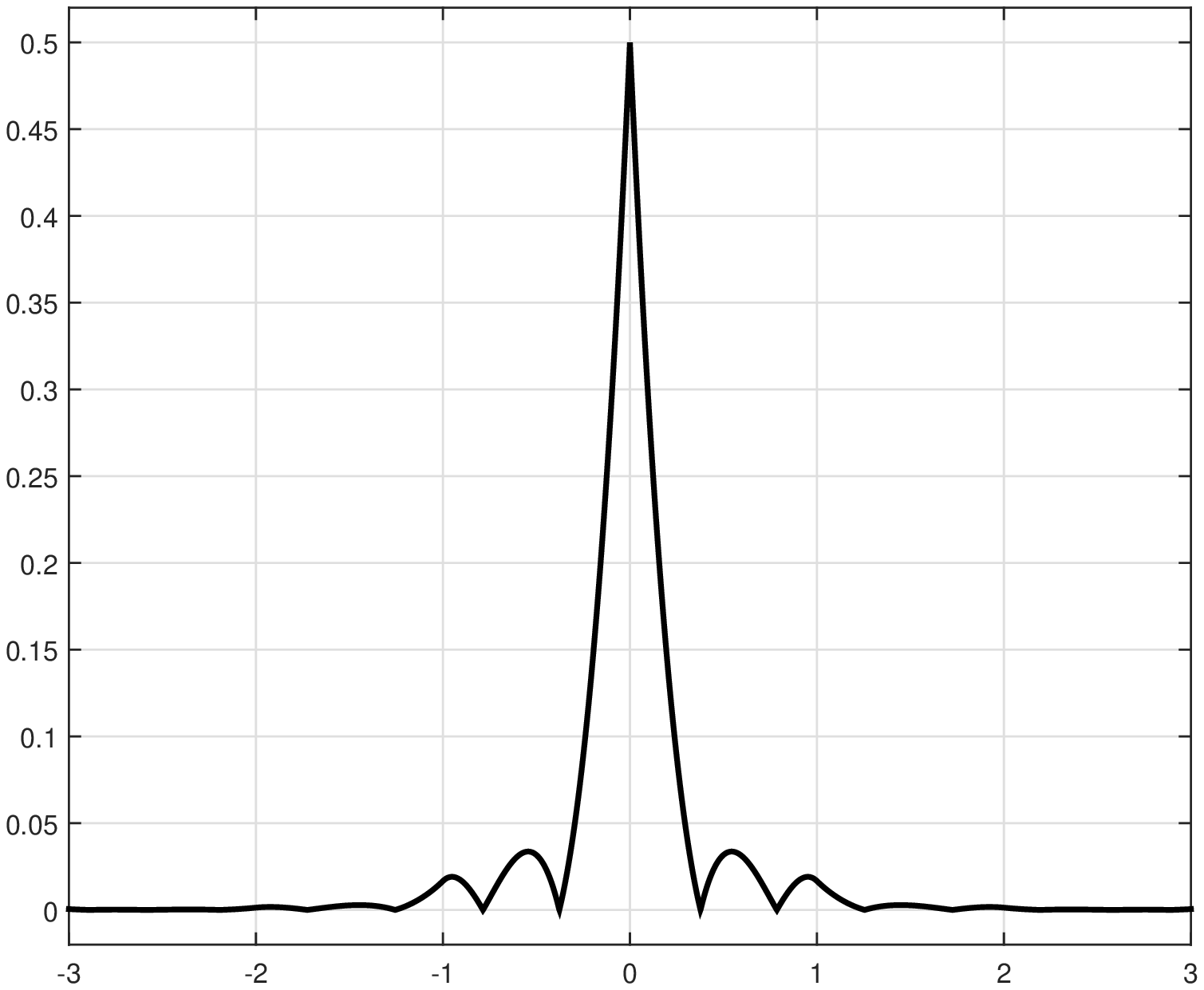}

    Pointwise error
  \end{minipage}
     \begin{minipage}{\linewidth}
    \centering
    \includegraphics[width=\linewidth]{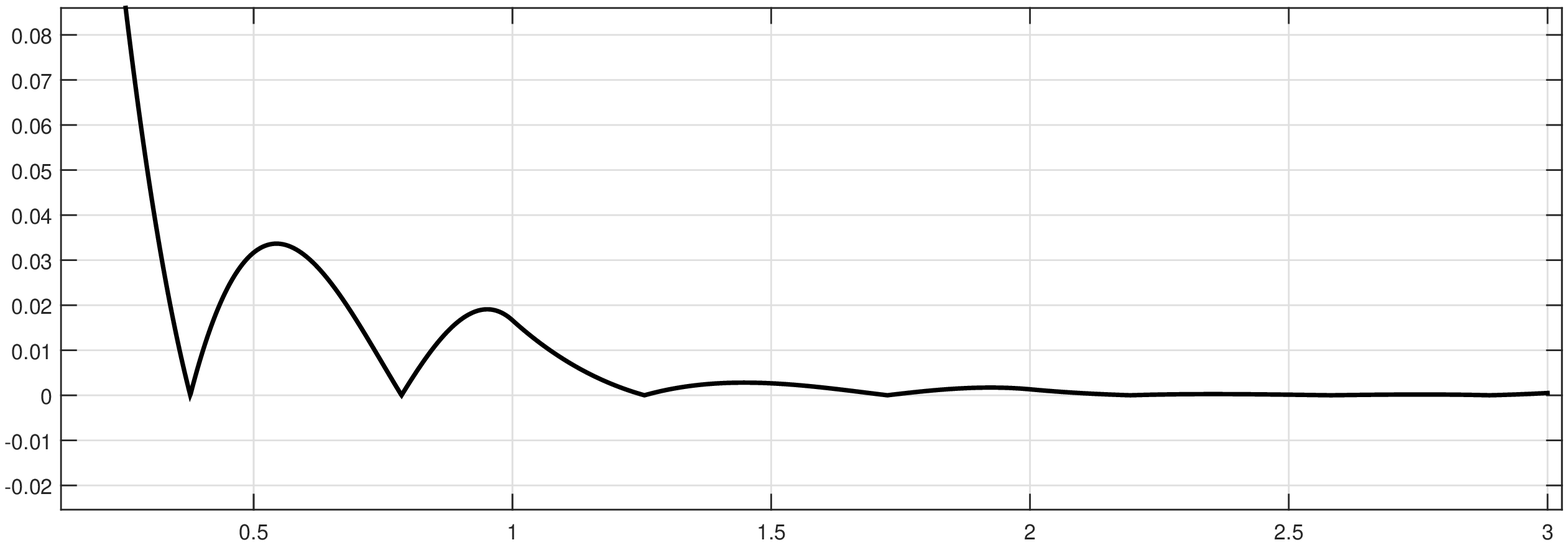}

    Zoom on the pointwise error
  \end{minipage}

  \caption{Best $L_1$ approximation of the Heaviside function with seven evenly spaced knots and
pointwise error graphs.}
\label{fig_7nodes}
\end{figure}

\begin{figure}[!h]
  \centering
  \begin{minipage}{0.45\linewidth}
    \centering
    \includegraphics[width=\linewidth]{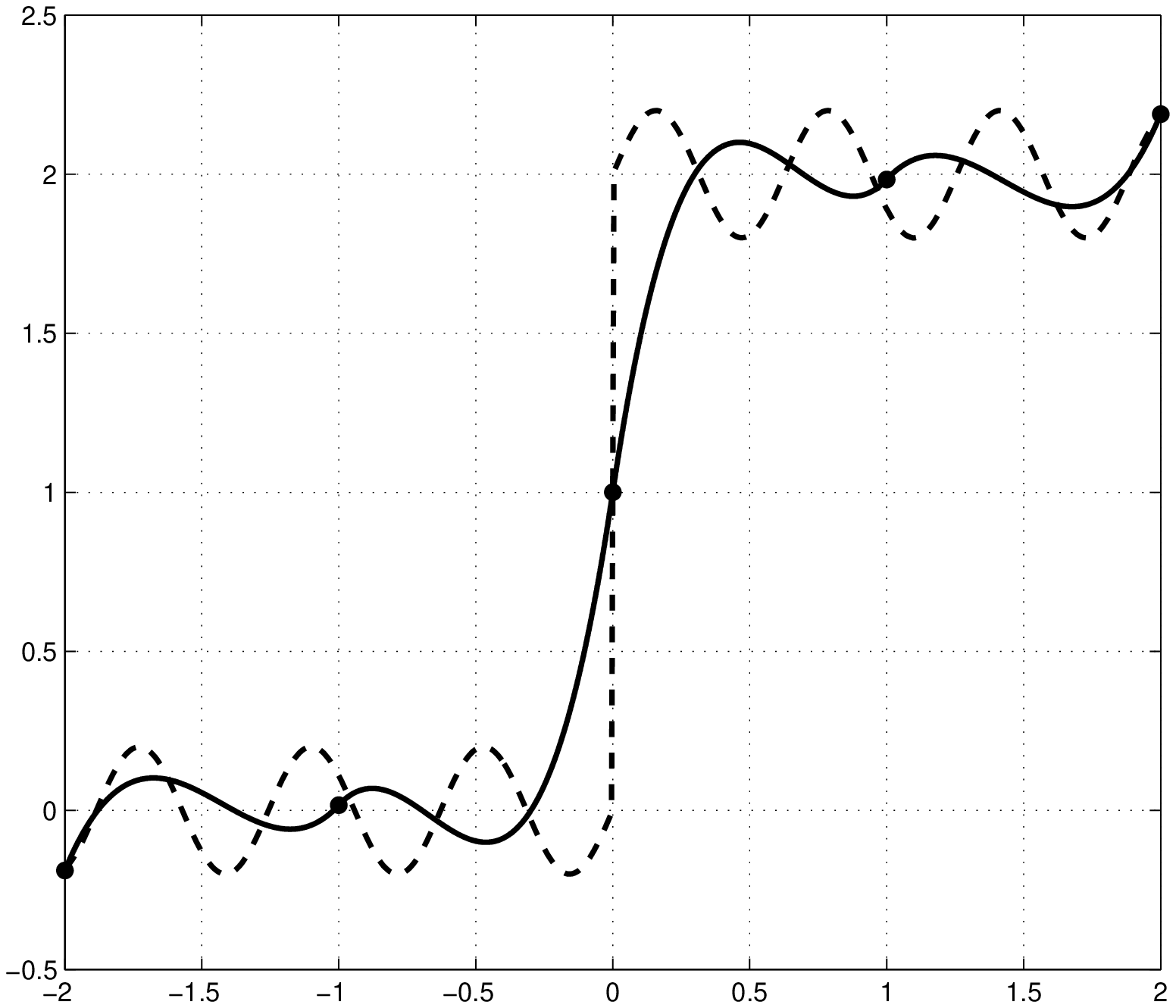}
  \end{minipage}
  \begin{minipage}{0.45\linewidth}
    \centering
    \includegraphics[width=\linewidth]{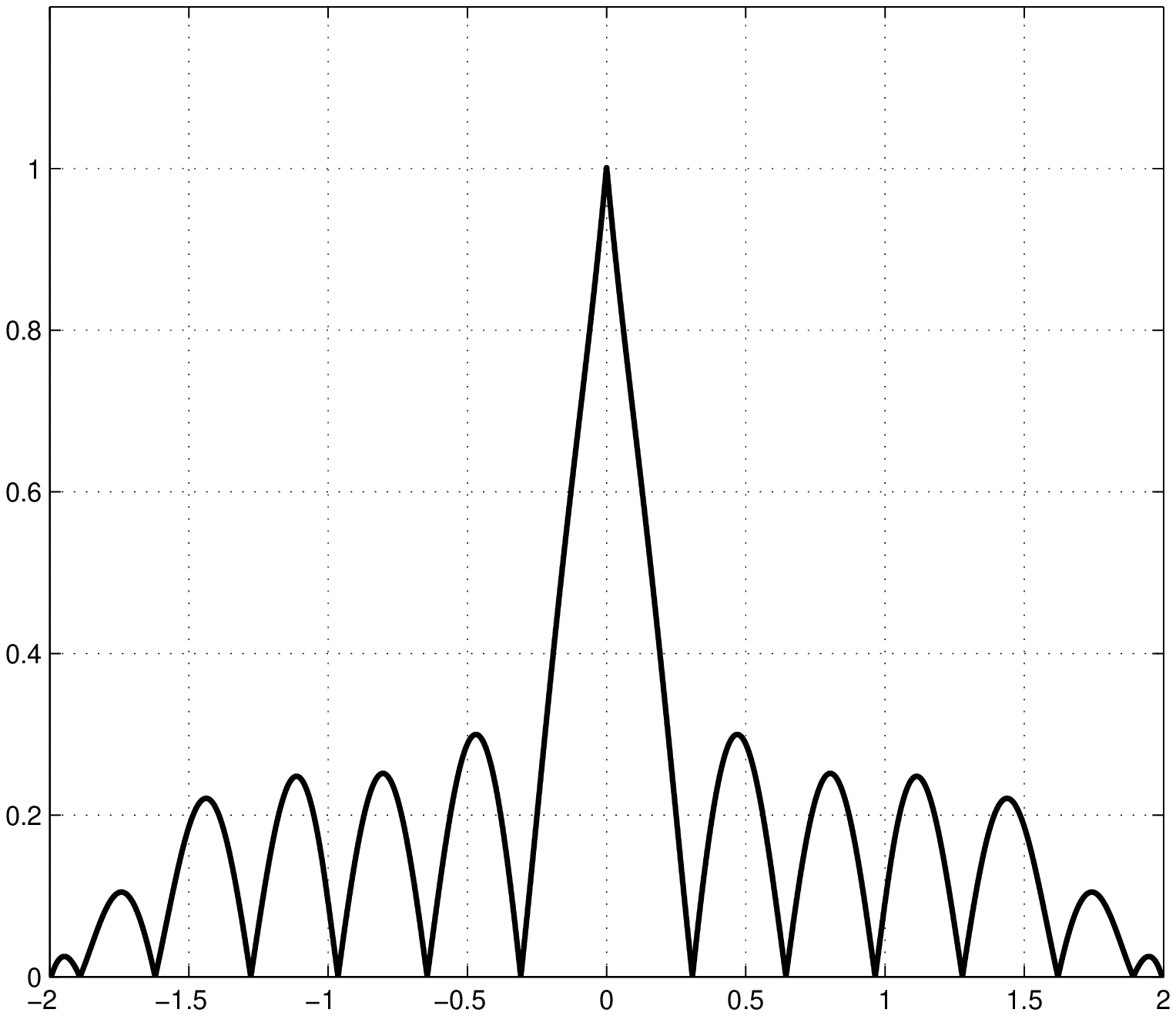}
  \end{minipage}
  \caption{Best $L_1$ approximation (solid line) of an oscillating Heaviside-type function (dashed line) with five evenly spaced knots and pointwise error graph.}
  \label{fig_oscillating_7nodes}
\end{figure}

\section{Conclusion}

In this article, we have studied the problem of best $L_1$ approximation of Heaviside-type functions in Chebyshev and weak-Chebyshev spaces. We have made a supplementary hypothesis on the dimension of the subspace composed of the even functions which is satisfied by some very classical spaces such as polynomials, trigonometric polynomials or spline functions with symmetrically distributed knots. Under that hypothesis, we have proved the uniqueness of best $L_1$ approximation of Heaviside-type functions in a Chebyshev space. Moreover, we have shown that for a class of Heaviside-type functions, the unique best $L_1$ approximation can be obtained by Lagrange interpolation at some abscissae determined in an adapted extension of the Hobby-Rice theorem. This result can be apply to compute best $L_1$ approximations in practical cases such as polynomial or trigonometric polynomial approximations.\\
We have given first results about best $L_1$ approximation of Heaviside-type functions in weak-Chebyshev spaces. We have seen that the minimal amount of intersections between the graphs of a best $L_1$ approximation and the one of its reference function is equal to $n-1$ where $n$ is the dimension of the weak-Chebyshev space. We have studied the practical case of Hermite polynomial splines and evidenced a Gibbs phenomenon for best $L_1$ approximation of the Heaviside function. Further work will focus on the determination of the abscissae in this case.
Future research will concentrate on best $L_1$ approximation of challenging Heaviside-type functions which require an asymmetric distribution of the real values in the extension of the Hobby-Rice theorem. The case of functions with multiple discontinuities is also on study.

%

\newpage
\bibliographystyle{siam}
\bibliography{BibArXiv}
\end{document}